\documentclass[12pt,letterpaper]{article}

\usepackage{graphicx}      

\usepackage{amsmath}
\usepackage{amsfonts}
\usepackage{mathtools}
\usepackage{makecell}
\usepackage{multirow}
\usepackage{geometry}

\usepackage{amsthm}
\usepackage{breqn}
\usepackage{accents}
\usepackage{enumerate}
\usepackage{amssymb}
\usepackage[utf8]{inputenc}
\usepackage{xparse}
\usepackage[acronym]{glossaries}

\newtheorem{thm}{Theorem}
\newtheorem{cor}{Corollary}[thm]
\newtheorem{lem}{Lemma}
\newtheorem{definition}{Definition}
\newtheorem{prop}{Proposition}

\newtheorem{example}{Example}

\newcommand{\R}{\mathbb{R}}
\newcommand{\N}{\mathbb{N}}
\newcommand{\union}{\bigcup}

\NewDocumentCommand{\ifs}{O{\mathcal{F}}}{\mathcal {#1}}
\NewDocumentCommand{\HB}{O{\operatorname{F}}}{\operatorname{#1}} 
\DeclarePairedDelimiter\roundBracket{\lparen}{\rparen}
\newcommand{\round}{\roundBracket*}
\DeclarePairedDelimiter\multiBracket{[}{]}
\newcommand{\multi}{\multiBracket*}
\NewDocumentCommand{\f}{o}{\IfNoValueTF {#1} {\operatorname{f}} {\operatorname{f}\round{#1}} }
\NewDocumentCommand{\dist}{o}{\IfNoValueTF{#1} {\operatorname{d}} {\operatorname{d}\round{#1}} }

\NewDocumentCommand{\closure}{m}{\overline{#1}}
 \newcommand{\ball}{\mathcal B}
  \newcommand{\ballbrack}{\round}

 \newcommand{\multito}{\rightrightarrows}
\newcommand{\opn}[1]{\operatorname{#1}}

\DeclarePairedDelimiter{\setbrack}{\{} {\}}
\newcommand{\set}{\setbrack*}

\NewDocumentCommand{\B}{o m m}{\IfNoValueTF {#1} {\ball_{#2} \ballbrack {#3}} {\ball^{#1}_{#2} \ballbrack {#3}}}

\NewDocumentCommand{\F}{o}{\IfNoValueTF {#1} {\operatorname{F}} {\operatorname{F}\multi{#1}} }

\NewDocumentCommand{\localB}{m}{B\round{#1}}
\NewDocumentCommand{\wLsB}{m}{\mathcal W \round {#1} }

\NewDocumentCommand{\reach}{o}{\IfNoValueTF {#1} {\operatorname{R}} {\operatorname{R} [#1] }}
\NewDocumentCommand{\chainreach}{o}{\IfNoValueTF {#1} {\operatorname{CR}}{\operatorname{CR} [#1]}}

\makeatletter

\makeglossaries

\newglossaryentry{ifs}
{
    name=IFS,
    description={A subset of all functions that map a set $X$ to it's self.}
}

\newglossaryentry{attractor}
{
    name=attractor,
    description={The attractive set of an \gls{ifs}.}
}

\newglossaryentry{strict attractor}
{
    name=strict attractor,
    description={The attractive set of an \gls{ifs}.}
}

\newglossaryentry{point wise attractor}
{
    name=point wise attractor,
    description={The attractive set of an \gls{ifs}.}
}

\newglossaryentry{semi attractor}
{
    name=semi-attractor,
    description={The attractive set of an \gls{ifs}.}
}

\newglossaryentry{quasi attractor}
{
    name= quasi attractor,
    description={The attractive set of an \gls{ifs}.}
}

\newglossaryentry{Conley attractor}
{
    name= Conley attractor,
    description={The attractive set of an \gls{ifs}.}
}

\newglossaryentry{small attractor}
{
    name= small attractor,
    description={The attractive set of an \gls{ifs}.}
}

\newglossaryentry{minimal attractor}
{
    name= minimal attractor,
    description={The attractive set of an \gls{ifs}.}
}

\newglossaryentry{sub-invariant}
{
    name=invariant,
    plural=invariance,
    description={The attractive set of an \gls{ifs}.}
}

\newglossaryentry{sup-invariant}
{
    name=super-invariant,
    plural=super-invariance,
    description={The attractive set of an \gls{ifs}.}
}

\newglossaryentry{invariant}
{
    name=invariant,
    description={The attractive set of an \gls{ifs}.}
}

\newglossaryentry{equicontinuous}
{
    name=equicontinuous,
    plural=equicontinuity,
    description={The attractive set of an \gls{ifs}.}
}

\newglossaryentry{evenly continuous}
{
    name=evenly continuous,
    plural=even continuity,
    description={The attractive set of an \gls{ifs}.}
}

\newacronym{cec}{c.e.c}{compositionally \gls{equicontinuous}}

\newacronym{elsc}{e.l.s.c}{evenly \gls{lower semi continuous}}

\newacronym{eusc}{e.u.s.c}{evenly \gls{upper semi continuous}}

\newacronym{tlse}{t.l.s.e}{topologically lower semi equicontinuous}

\newacronym{tuse}{t.u.s.e}{topologically upper semi equicontinuous}

\newglossaryentry{normal}
{
    name=disjunctive,
    description={A semi-infinite string containing every finite string as a substring.}
}

\newglossaryentry{lower semi continuous}
{
    name= lower semicontinuous,
    plural= lower semicontinuity,
    description={The attractive set of an \gls{ifs}.}
}

\newacronym{lsc}{l.s.c.}{\gls{lower semi continuous}}

\newglossaryentry{upper semi continuous}
{
    name= upper semicontinuous,
    plural= upper semicontinuity,
    description={The attractive set of an \gls{ifs}.}
}

\newacronym{usc}{u.s.c.}{\gls{upper semi continuous}}

\newglossaryentry{outer semi continuous}
{
    name= outer semicontinuous,
    plural= outer semicontinuity,
    description={The attractive set of an \gls{ifs}.}
}

\newacronym{osc}{o.s.c.}{\gls{outer semi continuous}}

\newglossaryentry{multifunction}
{
    name=multifunction,
    description={A function from $X$ to $2^Y$.}
}

\newglossaryentry{HBO}
{
    name=Hutchinson-Barnsley operator,
    description={A function from $X$ to $2^Y$.}
}

\newglossaryentry{complete domain}
{
    name=correspondence,
    description={A function from $X$ to $2^Y$.}
}

\title{A Necessary Condition on Chain Reachable Robustness of Dynamical Systems}

\author{Maxwell Fitzsimmons, Jun Liu\footnote{This work was supported in part by the NSERC DG, CRC, and ERA programs.}}

\begin{document}

\maketitle


\begin{abstract}                
It is ``folklore'' that the solution to a set reachability problem for a dynamical system is only noncomputable  because of non-robustness reasons. A robustness condition that can be imposed on a dynamical system is the requirement of the chain reachable set to equal the closure of the reachable set. We claim that this condition necessarily imposes strong conditions on the dynamical system.
For instance, if the space is connected and compact and we are computing a chain reachable robust single valued function $\f$ then $\f$ cannot have an unstable fixed point or unstable periodic cycle.
\end{abstract}


\section{Introduction}
Many problems in control theory can be solved immediately if one has access to the reachable set of a dynamical system. Unfortunately, it is often difficult to exactly compute the reachable set. However, there are many algorithms that give approximations to the reachable set, for example see \cite{gan2017reachability,rungger2018accurate,fan2016automatic,gao2016interpolants,lal2019counterexample}. 

As exact computation of the reachable set is difficult, researchers have investigated this problem through the lens of computability theory, see \cite{collins2005continuity,collins2007optimal, bournez2010robust,chen2015continuous,fijalkow2019decidability,kong2015dreach,akshay2016decidable}. In fact the reachable set of a general dynamical system (both in discrete time and continuous time) is noncomputable. This means that we need to find conditions on a dynamical system in order for the reachable set to be computable.

It is generally believed that the reachable set is not computable due to the dynamical system being ``non-physical'' or``artificial''; the dynamical system is some mathematical oddity that would never arise in a practical situation. Informally, we may say the dynamical system is somewhat robust if its reachable set is computable. In this work we will examine the implications of a discrete-time dynamical system being chain reachable robust (intuitively, the dynamics are insensitive to infinitesimal perturbations) first examined in \cite{collins2005continuity}; where the authors show that if a dynamical system is chain reachable robust then the reachable set is computable.
In fact the authors showed, in their framework of computability, that chain reachable robustness was also a necessary condition on computability of the reachable set. This robustness condition was also used in \cite{bournez2010robust} to prove certain continuous-time dynamical systems have computable reachable sets.

Although the computability result in \cite{collins2005continuity} is sharp, this work (and other work in the literature to the best of our knowledge) provided no practically verifiable sufficient conditions (or any for that matter) for a dynamical system to be chain reachable robust. Our original intention for this paper was to provide at least one non-trivial practically verifiable sufficient condition for a dynamical system to be chain reachable robust. We have failed in this regard. Instead, we provide a necessary condition on chain reachable robustness and assert that this necessary condition is likely too strong of a condition for practical purposes. More specifically, we claim that chain reachable robustness imposes strong conditions on the long-term behavior of the dynamics. Our main result, Theorem \ref{thm:unique or infinite}, states that the long-term behavior of a chain reachable robust system (in a connected compact metric space) is always stable and that the number of ``long-term behaviors\footnote{By this we are referring to minimal sets, see Subsection \ref{subsec:minimal sets}.}'' is either one or infinity. In the case where $\f: X \to X$ is a continuous function, $X$ is connected compact set, the dynamics are $x_n = \f[x_{n-1}]$, and the system is chain reachable robust, then all of the fixed points and periodic cycles of $\f$ are stable. In the case where there is a unique fixed point (or periodic cycle), it is globally asymptotically stable.

In Section \ref{sec:prelim} we briefly introduce necessary background information concerning chain reachable robustness, multifunctions and minimal sets. In Section \ref{sec:ness} we develop several technical results about the reachable set (largely under the assumption the system is chain reachable robust) to prove Theorem \ref{thm:unique or infinite}.


\section{Preliminaries}\label{sec:prelim}
For simplicity we will work in metric spaces, rather than topological spaces like in  \cite{collins2005continuity,collins2007optimal}. We will consider discrete time dynamical systems  with control and without. Let $\round{X,\dist}$ be a metric space, $U$ be a set, and $\f: X \times U \to X$ be a function. Let the dynamics be
\begin{equation}\label{eqn:control dynamics}
x_{n+1} = \f[x_n,u_n]
\end{equation}
for some $\set{u_n}_{n \in \N} \subseteq U$. Another way to write the above is to define a multifunction $\F: X \multito X$ by $\F[x] = \f[x,U]$ and the dynamics are $x_{n+1} \in \F[x_n]$.
If we wish to not use control, then will simply write $x_{n+1} = \f[x_n]$.


\begin{definition}\label{def:reach}
Let $\round{X,\dist}$ be a metric space, $C \subseteq X$, and $\F: X \multito X$ be a multifunction. Define the reachable set
\begin{align*}
\reach[\F,C]=\{& x \in X: \exists \set{x_n}_{n=0}^N, N\geq 0, \text{ s.t. }
 x_{i} \in \F[x_{i-1}], 1 \leq i \leq N ,x_0 \in C,&  \text{and } x=x_N  \}.
\end{align*}
If the multifunction is understood, we may instead write $\reach[C]$ to be the reachable set.
\end{definition}
We can see that $\reach[\F,C] = \union_{n=0}^\infty \F^{\circ n} \multi C,$ where $\F[C] = \union_{c \in C} \F[c]$, $\F^{\circ 0}[x] = \set{x}$, and $\F^{\circ n}\multi x = \F[\F^{\circ (n-1)}\multi x ]$.

For $\epsilon >0$ and a set $A \subseteq X$, we use the notation  $A_\epsilon = \B \epsilon A =\union_{a \in A} \B \epsilon a$, where $\B \epsilon x = \set{y\in X : \dist[x,y]<\epsilon}$.
\begin{definition}\label{def:chain reach}
Let $\round{X,\dist}$ be a metric space, $C \subseteq X$, and $\F: X \multito X$ be a multifunction.
Let $\epsilon >0$, we define an $\epsilon$-chain of $\multi{\F,C}$ to be $\set{y_n}_{n=0}^N$, $N \geq 0$, with $y_{i} \in \F_\epsilon\multi {y_{i-1}} :=  \B \epsilon {\F[y_{i-1}]}$, $1 \leq i \leq N$, and $y_0 \in C$.

Define the chain reachable set
\begin{multline*}
\chainreach[\F,C] = \{x \in X:  \forall \epsilon >0, \exists \set{y_n}_{n=0}^N, \text{ an } \epsilon\text{-chain of } \multi{\F,C}, \text{ s.t } x=y_N\}.
\end{multline*}
If the multifunction is understood, we may instead write $\chainreach[C]$ to be the chain reachable set.
The reachable set $\reach[\F,C]$ is said to be chain reachable robust or simply robust if $\closure{\reach[\F,C]} =\chainreach[\F,C] $.
\end{definition}
The chain reachable set is closed assuming that $\f$ is continuous in both its variables and $U$ is a compact set. In view of (\ref{eqn:control dynamics}), an $\epsilon$-chain of $\multi{F,C}$ is also of the form: $\set{y_n}_{n=0}^N$, $N \geq 0$, and
\[
\dist[y_{i},\f[y_{i-1},u_{i-1}]] < \epsilon,
\]
for $1 \leq i \leq N$, where $\set{u_n}_{n=0}^{N-1} \subseteq U$, and $y_0 \in C$.
Additionally, if we define $\F_\epsilon \multi {x} = \B \epsilon {\F[x]}$, then $\chainreach[\F,C]  = \bigcap_{\epsilon >0} \union_{n=0}^\infty \F^{\circ n}_\epsilon \multi C$.
In \cite{collins2007optimal} the authors showed that the chain reachable set is an optimal over-approximation of the reachable set.

The idea of using $\epsilon$-chains or perturbed dynamics to study the true dynamics is widely used in verification and control of dynamical systems, for example see \cite{kong2015dreach, liu2017robust, li2018rocs, li2018robustly}.

\subsection{Multifunctions}\label{subsec:multi}
A multifunction from $X$ to $Y$ is a function from $X$ to $2^Y\setminus \emptyset$. If $\F$ is a multifunction from $X$ to $Y$, we write $\F: X \multito Y$ and, for all $S \subseteq X$, we define $\F[S]= \union_{s \in S} \F[s]$.

\begin{definition}\label{def:preimage}
Let $X,Y$ be sets and $\F : X \multito Y$. Define, for all $B\subseteq Y$,
the upper pre-image of $\F$ as
\[
\F^+ \multi{B} = \set{x\in X : \F[x] \subseteq B}.
\]
and the lower pre-image of $\F$ as
\[
\F^- \multi{B} = \set{x\in X : \F[x] \cap B\neq \emptyset}.
\]
\end{definition}
Often, the lower pre-image is called the inverse {multifunction} of $\F$; note that $\F^-$ is a  {multifunction} in its own right, while $\F^+$ is not.

\begin{definition}\label{def:semi continuity}
Let $\round{X,\tau},\round{Y,\rho}$ be topological spaces, and $\F : X \multito Y$. We say that $\F$ is \acrfull{lsc} if, for all $V$ open in $Y$, $\F^-\multi V$ is open in $X$. We say that $\F$ is \acrfull{usc} if,  for all $V$ open in $Y$  $\F^+\multi V$ is open in $X$.

 If $\F$ is both lower and upper semicontinuous, then we call $\F$ continuous. 
\end{definition}

We would like to note that, in \cite{collins2005continuity,collins2007optimal}, they assume that the multifunctions being computed are closed-valued continuous multifunctions.

\begin{prop}\label{prop:multi facts}
Let $\round{X,\dist},\round{Y,\rho}$ be metric spaces and $\F : X \multito Y$. Then
\begin{enumerate}
\item $\F$ is \acrshort{lsc} if and only if, for all $S \subseteq X$, we have $\F[\closure{S}] \subseteq \closure{\F[S]}$. \label{item:lsc}
\item Assume that $\F$ is compact-valued. Then $\F$ is \acrshort{usc} if and only if, for every compact set $C \subseteq X$ and every $\epsilon >0$, there is a $\delta >0$ such that $ \F[C_\delta] \subseteq \F_\epsilon \multi C$.
\item Assume that $\F$ is compact-valued. Then  $\F$ is \acrshort{usc} if and only if, for every point $x \in X$ and every $\epsilon >0$, there is a $\delta >0$ such that $ \F[\B \delta x] \subseteq \F_\epsilon \multi x$. \label{item:point usc}

\item{ $\F$ is \acrshort{usc} if and only if, for every closed set $C \subseteq Y$ we have that $\F^{-} \multi C$ is closed.} \label{item:usc closed}

\item{If $\F$ is \acrshort{lsc}, then $\reach[F,x]$ and $\closure{\reach[F,x]}$ are \acrshort{lsc} multifunctions of $x$.} \label{item:reach lsc}
\end{enumerate}

\end{prop}

\begin{proof}
The proof of items (1) through (4) can be found in Chapter 1 of~\cite{hu1997handbook}. The proof of item (5) follows from verifying item (1) holds for the multifunctions in question. ~
\end{proof}

If we have two multifunctions $\F: X \multito Y$ and $\opn{G}: Y \multito Z$, define the composition multifunction $\opn{G}  \circ \F: X \multito Z$ by $\opn{G}  \circ \F \multi x = \opn{G} \multi{\F[x]}$. The composition of \acrshort{lsc} (\acrshort{usc}) multifunctions is again \acrshort{lsc} (\acrshort{usc}). Suppose that $P$ is a property sets can have (i.e. closed, open, convex, finite etc.). We say $\F$ is $P$-valued if, for all $x \in X$, $\F[x]$ has the property $P$. Instead of saying $\F$ is singleton-valued we will say $\F$ is single/point-valued. We define $\closure{\F}= \opn{clF}: X \multito Y$ to be $\closure{\F} \multi x= \opn{clF}\multi x = \closure{\F[x]}$ for all $x \in X$.

 Note both the chain reachable sets and reachable sets are multifunctions for a fixed $\F:X \multito X$. In this case $\reach,\chainreach:X \multito X$, $\reach[x] =\union_{n=0}^\infty \F^{\circ n} \multi x$, and $\chainreach[x] = \bigcap_{\epsilon >0} \reach[\F_\epsilon, x]$.

\subsection{Minimal Sets}\label{subsec:minimal sets}
Suppose that $X$ is a metric space and $\F: X \multito X$ is a multifunction. Then a set $ A \subseteq X$ is said to be a minimal set of $\F$, or simply a minimal set, if it is a minimal closed, nonempty, \gls{sub-invariant} set of $\F$. That is, $A$  is closed, nonempty and satisfies $\F[A] \subseteq A$. Further, for all $B \subseteq A$ that is closed, nonempty and satisfies $\F[B] \subseteq B$, we must have that $B=A$. In a compact space with $\F= \set{\f}$ being single-valued, a minimal set is where all the long-term behavior of the sequence $\set{\f^{\circ n}}_{n \in \N}$ ``happens''. In this section we state a number of results about minimal sets of a \acrshort{lsc} multifunction.

\begin{prop}\label{prop:minimal sets}
Let $\round{X,\dist}$ be a metric space, $A \subseteq X$ be a set, and  $\F: X \multito X$ be a \acrshort{lsc} multifunction. Then the following are equivalent:
\begin{enumerate}
\item $A$ is a minimal set of $\F$.
\item $A \neq \emptyset$ and for all $a \in A$ we have $\closure{\reach[\F,a]} = A$.
\end{enumerate}
Furthermore, if $\closure{\reach[\F,x]}$ is compact for some $x \in X$, then there is a compact minimal set $A \subseteq \closure{\reach[\F,x]}$.

\end{prop}

\begin{proof}
The equivalence of items (1) and (2) follows from the observation that $\closure{\reach[\F,x]}$ is a nonempty closed invariant set of $\F$. This fact follows from item (\ref{item:lsc}) of Proposition \ref{prop:multi facts}.

To prove the ``furthermore'', one can apply (the dual of) Zorn's Lemma to the set
\[
\set{B \subseteq \closure{\reach[x]} : \emptyset \neq B \text{ is closed and \gls{sub-invariant}}}
\]
equipped with the partial order $\subseteq$. ~
\end{proof}

In the case that $\F=\set{\f}$ is single-valued, minimal sets are typically fixed points of $\f$ (even in the multi-valued case, we have $\closure{\F[A]} =A$ if $A$ is minimal) or limit cycles of $\f$, one of which must be the case if the minimal set is finite.

\begin{example}\label{ex:minimal sets}
Let $X$ be the unit circle in the complex plane with the usual metric. Every point in $X$ can be uniquely represented in the form $e^{2\pi i x}$, where $x \in [0,1)$ and $i^2 =-1$. Define the map
\[
\f[e^{2\pi i x}] = e^{2\pi i \round {x + \theta}}
\]
for $x,\theta \in [0,1)$ and $\f:X \to X$. If $\theta= \frac{p}{q}$ for $p,q \in \mathbb Z, q \neq 0$ and $p,q$ are relatively prime, then the minimal sets of $\f$ are all of the form $\set{z,\f[z],\dots, \f^{\circ q}\round z}$, where $z$ could be any point in $X$. In fact, every point in $X$ belongs to a minimal set.
If $\theta$ is irrational, then the unique minimal set of $\f$ is $X$ (this follows from the relatively well-known fact that the sequence $\set{(x +n\theta)\opn{mod} 1}_{n \in \N}$ is dense on $[0,1]$ when $\theta$ is irrational). This is an example of a minimal set that is not a fixed point or periodic cycle.
\end{example}

\begin{definition}\label{def:Lyapunov stable}
Let $\round{X,\tau}$ be a topological space, $A \subseteq X$ be a set, and  $\F: X \multito X$ be a multifunction. Then $A$ is said to be Lyapunov stable if, for every open set $V \supseteq A$, there is a open set $W \supseteq A$ with $\reach[W] \subseteq V$.
\end{definition}
A Lyapunov stable minimal set is the place where the long-term behavior of the dynamics from Equation~(\ref{eqn:control dynamics}) happens, assuming that the dynamics reach the minimal set in the long-term.

\begin{prop}\label{prop:Lyapunov stable orbit}
Let $\round{X,\dist}$ be a metric space space, $U$ be a set, and $\f : X \times U \to X$ be a function such that,  for all $u \in U$, we have that $\f\round{\cdot,u}=\f_u \round{\cdot}$ is continuous.
Furthermore, let $A$ be a Lyapunov stable compact set and $\set{x_n}_{n \in \N}$ be a sequence defined by Equation~(\ref{eqn:control dynamics}) with $\closure{\set{x_n}_{n \in \N}}$ compact. Then we have
\[
\closure{\set{x_n}_{n \in \N}} \cap A \neq \emptyset \implies \bigcap_{N \in \N} \closure{\set{x_n}_{n =N}^\infty} \subseteq A
\]
and, in the case where $U$ is singleton (no control) and $A$ is a minimal set of $\F[x] = \set{\f\round x}$ where $\f$ is continuous, we have $\closure{\set{x_n}_{n \in \N}} \cap A \neq \emptyset \implies \bigcap_{N \in \N} \closure{\set{x_n}_{n =N}^\infty} = A $.

\end{prop}

\begin{proof}
We claim that, for every open set of $V \supseteq A$, there is an $N \in \N$ for all $n\geq N$ such that $x_n \in V$, provided $\closure{\set{x_n}_{n \in \N}} \cap A\neq \emptyset$. To see this, pick $a \in \closure{\set{x_n}_{n \in \N}} \cap A$ and any open $V \supseteq A$. Then, by  Lyapunov stability of $A$ there is $W\supseteq A$ such that $\reach[W] \subseteq V$. As $a \in \closure{\set{x_n}_{n \in \N}}$ and $W$ is an open set of $a \in A$, we have  $W \cap \set{x_n}_{n \in \N} \neq \emptyset$. So there is $N \in \N$ with $x_N \in W$. But for every $n > N$ we have
\begin{align*}
x_n \in \f[x_{n-1},U] = \F[x_{n-1}] &\subseteq \F^{\circ \round{n-N}}\multi {x_N} \\ &\subseteq \reach[W] \subseteq V.
\end{align*}
This proves the claim.

Now, the set $\bigcap_{N \in \N} \closure{\set{x_n}_{n =N}^\infty}$ is the limit points of the convergent subsequences of $\set{x_n}_{n \in \N}$.  So suppose, for the sake of contradiction, that $y \in X\setminus A$ is a limit of a subsequce of $\set{x_n}_{n \in \N}$. Then, since $A$ is compact, there are open sets $V \supseteq A$ and $O \ni y$ with $V\cap O = \emptyset$, but by the claim the sequence is eventually in $V$, so it cannot eventually be in $O$. Hence, $y$ cannot be a limit point of the sequence, a contradiction, and so $\bigcap_{N \in \N} \closure{\set{x_n}_{n =N}^\infty} \subseteq A$.

In the case where $A$ is minimal and $\F=\set{\f}$ single-valued, the set $\bigcap_{N \in \N} \closure{\set{x_n}_{n =N}^\infty}  =\bigcap_{N \in \N} \closure{\set{\f^{\circ n} \round x}_{n =N}^\infty}$ is closed, nonempty (by compactness), and \gls{sub-invariant}. 
The set in question is contained in $A$ by the first part of this theorem and by minimality we must have $\bigcap_{N \in \N} \closure{\set{\f^{\circ n} \round x}_{n =N}^\infty}=A$. ~
\end{proof}

Effectively we know that if the dynamics ``touch'' a Lyapunov stable set  we know  the long-term behavior (the limit points of the dynamics) must also be in this Lyapunov stable set. At this point it is natural to ask when a set is Lyapunov stable.

\begin{prop}\label{prop:usc stable}
Let $\round{X,\dist}$ be a metric space, $\F: X \multito X$ be a multifunction, and $A \subseteq X$ be a compact \gls{sub-invariant} set. If $\reach$ or $\closure{\reach}$ is \acrshort{usc}, then $A$ is Lyapunov stable. In particular, every compact minimal set is Lyapunov stable.
\end{prop}

\begin{proof}
The result follows from noticing that $\reach^+\multi V$ and $\opn{clR}^+\multi V$ are \gls{sub-invariant} neighborhoods of $A$ whenever $V\supseteq A$ and $A$ is \gls{sub-invariant}. ~
\end{proof}

Later, we will use some results about the set
\begin{equation}\label{eqn:weak basin}
\wLsB A = \set{x \in X : A \subseteq \closure{\reach[\F,x]}}.
\end{equation}

\begin{thm}\label{thm:weak basin}
Let $\round{X,\dist}$ be a metric space, $\F: X \multito X$ be a \acrshort{lsc} \gls{multifunction}, and  $A$ be a minimal set of of $\F$.
Then the following holds for any local basis $\localB{a}$, $a \in A$ (a local basis of a point $x$ is a collection of sets with the following property: for any open $V \ni x$, there is a $U \in \localB x$ with $x \in \opn{int}\round{U} \subseteq U \subseteq V$):
\begin{enumerate}
\item{\(\wLsB A= \opn{clR}^{-} \multi A 
\). }
\item {For any $a \in A$, $$\wLsB A= \bigcap_{V \in \localB a} \reach^{-} \multi {V }.$$}
\item{
\(
\opn{clR}^{-} \multi {\wLsB A} =\wLsB A.
\)
}
\item $\wLsB A$ is open if and only if $\wLsB A$ is a neighborhood of some $a \in A$.
\item If $\closure{\reach}=\opn{clR}$ is \acrshort{usc} then $\wLsB A$ is closed.
\end{enumerate}
\end{thm}

\begin{proof}
Items~(1), (2) and (3) are shown in Theorem~15 of~\cite{thesis}.
To prove item~(4), notice that $A \subseteq \wLsB A$ (this can be seen from item~(1) and the fact $A$ is \gls{sub-invariant}). So if $\wLsB A$ is open, it must be a neighborhood of a point of $A$.

 Conversely, if for some $a \in A$ there is open $V \ni a$ with $ V \subseteq \wLsB A$, then we can apply $\reach^{-}$ to both sides of this to yield
 \[
 \reach^{-} \multi V \subseteq \reach^{-} \multi {\wLsB A} \subseteq \opn{clR}^{-} \multi {\wLsB A} = \wLsB A
 \]
 by item~(4). By item~(3) and taking $\localB a$ to be the set of all open neighborhoods of $a$, we have $\wLsB A= \bigcap_{W \in \localB a} \reach^{-} \multi {W} \subseteq  \reach^{-} \multi V $. Hence, $\wLsB A = \reach^{-} \multi V $ and since $\reach$ is \acrshort{lsc}, $\reach^{-} \multi V $ is open.

Item (5) is trivial when one recalls item~(\ref{item:usc closed}) of Proposition \ref{prop:multi facts}. Hence, for every closed set $C\subseteq X$ we have that $\opn{clR}^{-} \multi C$ is closed. So we apply item~(1) of this theorem and recall that $A$ is closed to conclude the result. ~
\end{proof}

\section{Necessary Conditions on Robustness in Compact Spaces}\label{sec:ness}

For the purposes of this section, we will typically be working in a connected compact metric space $X$ and considering a robust multifunction $\F$; we will call $\F$ robust if for all $x \in X$ the set  $\reach[\F,x]$ is robust; that is, $\closure{\reach[\F,x]}=\chainreach[\F,x]$. From a mathematical point of view, this ends up being a strong condition.

\begin{lem}\label{lem:characterized robust}
Let $\round{X,\dist}$ be a compact metric space and $\F:X \multito X$ be a multifunction. Then for all $x \in X$:

 $\chainreach[x] =\closure{\reach[x]}$ if and only if, for every $\epsilon >0$, there is a $\delta>0$ for which
\[
\reach[\F_\delta,x] = \union_{n=0}^\infty \F_\delta^{\circ n} \multi x \subseteq  \reach_\epsilon \multi x.
\]
\end{lem}

\begin{proof}
To begin, we claim that $\chainreach[x] = \bigcap_{\delta >0} \closure{\reach[\F_\delta,x]}$. Since
\[
\chainreach[x] =  \bigcap_{\epsilon >0} \reach[\F_\epsilon,x],
\]
the $\subseteq$ inclusion is immediate.
Thus suppose that $y \in \bigcap_{\delta >0} \closure{\reach[\F_\delta,x]}$ and $\epsilon >0$ is arbitrary then $ y \in \closure{\reach[\F_{\frac{\epsilon}{2}},x]} = \closure{ \union_{n=0}^\infty \F_{\frac{\epsilon}{2}}^{\circ n} \multi x}$ and so $\B {\frac{\epsilon}{2}} y \cap { \union_{n=0}^\infty \F_{\frac{\epsilon}{2}}^{\circ n} \multi x} \neq \emptyset$.
It follows that $y \in \B {\frac{\epsilon}{2}} {\F^{\circ n}_{\frac{\epsilon}{2}} \multi x} =\F_{{\epsilon}} \multi {\F^{\circ n-1}_{\frac{\epsilon}{2}}\multi x} $ for some $n$ and since every ${\frac{\epsilon}{2}}$-chain is an $\epsilon$-chain, we have that $\F^{\circ n-1}_{\frac{\epsilon}{2}} \subseteq \F^{\circ n-1}_{\epsilon} $.
Hence, $y \in{\union_{n=0}^\infty \F_{{\epsilon}}^{\circ n} \multi x}$ for all $\epsilon >0$ and  $\chainreach[x] = \bigcap_{\delta >0} \closure{\reach[\F_\delta,x]}$.

Suppose that there is $\epsilon >0$ for all $\delta >0$ such that $\reach[\F_\delta,x]  \cap X\setminus \reach_\epsilon \multi x \neq \emptyset$. Define $A^\delta = \closure{\reach[\F_\delta,x]}\cap X\setminus \reach_\epsilon \multi x$. Then the family of sets $\set{A^\delta: \delta >0}$ is a family of compact sets with the finite intersection property, because the sets are nested and nonempty. Thus,
\begin{align*}
\emptyset \neq \bigcap_{\delta >0} A^\delta &=  \round{\bigcap_{\delta >0} \closure{\reach[\F_\delta,x]}}\cap X \setminus \reach_\epsilon \multi x \\
&= \chainreach[x] \cap X \setminus \reach_\epsilon \multi x
\end{align*}
by the above claim.
We see that $\emptyset \neq \chainreach[x] \cap X\setminus \reach_\epsilon \multi x \subseteq \chainreach[x] \cap X\setminus \reach \multi x$ and so $\chainreach[x] \neq \reach[x]$.

Conversely, suppose that for all $\epsilon >0$ there is a $\delta >0$ with $\reach[\F_\delta,x] \subseteq \reach_\epsilon \multi x$. It can be shown that $\bigcap_{\delta >\eta>0} \reach[\F_\eta,x] = \chainreach[\F,x]$ and so we see that $  \chainreach[\F,x] \subseteq \reach_\epsilon \multi x$ for all $\epsilon>0$.
Since $ \bigcap_{\epsilon >0} \reach_\epsilon \multi x = \closure{\reach[x]}$, we have $\chainreach[x] = \closure{\reach[x]}$ (noting that the inclusion $\chainreach[x] \supseteq \closure{\reach[x]}$ always holds).
~ 
\end{proof}

It's unclear to us how to interpret of the $\epsilon$-$\delta$ condition in the above lemma. 
Certainly, the condition has implications on safety problems. We say $\multi{\F,x}$ is safe if $\reach[\F,x] \subseteq S$ where $S \subseteq X$ is interpreted as a ``safe'' set. If the $\epsilon$-$\delta$ condition is satisfied for $\multi{\F,x}$ and, for some $\epsilon >0$, we have that $\reach_\epsilon \multi{x} \subseteq S$ (i.e., $\multi{\F,x}$ is $\epsilon$-safe), then the $\delta$-perturbed system $\multi{\F_\delta,x}$ is safe, since $ \reach[{\F_\delta,x}] \subseteq \reach_\epsilon \multi{x} \subseteq S$. That is, $\epsilon$-safety of $\multi{\F,x}$ implies safety of $\multi{\F_\delta,x}$ for some $\delta>0$. This contrasts to a common use of these $\delta$-perturbed systems: since $\reach[\F,x] \subseteq \reach[\F_\delta,x]$ for any $\delta >0$, if $\multi{\F_\delta,x}$ is safe, then $\reach[\F,x]$ is safe. In other words,  $\delta$-perturbed systems can be used to determine the safety of the real system. In contrast, for chain reachable robust systems, the safety of the  $\delta$-perturbed system is guaranteed by the $\epsilon$-safety of the real system.

\begin{lem}\label{lem:discount usc}
Let $\round{X,\dist}$ be a metric space and $\F: X \multito X$ be a \acrshort{usc} multifunction. 
Then for all $\epsilon >0$ and all compact sets $C \subseteq X$, there is $\delta >0$ for all  $n \in \N$ such that
\[
 \F_\delta^{\circ n} \multi {\B \delta C} \subseteq \F_\epsilon^{\circ n} \multi {C}.
\]
Note that $\F_\epsilon^{\circ n}$ is the $n$-fold composition of $\F_\epsilon$ and not the epsilon enlargement of $\F^{\circ n}$.
\end{lem}

\begin{proof}
Given $\epsilon >0$, we will construct $\delta$ for $n=1$ and then proceed by induction. Since $\F$ is \acrshort{usc} and $C$ is compact, we get $\delta_1 >0$ so that
\[
 \F \multi {\B {\delta_1} C} \subseteq \B {\frac{\epsilon}{2}} {\F[C]} = \F_{\frac{\epsilon}{2}} \multi {C}.
\]
Now pick $\delta = \min\set{\delta_1, {\frac{\epsilon}{2}}}$ and consider $y \in  \F_\delta \multi{\B \delta C}$. So there is a $y' \in \F[\B \delta C] \subseteq \F_{\frac{\epsilon}{2}} \multi C $ with $\dist[y,y'] < \delta$. Hence

\[
\dist[y, \F[C] ] \leq \dist[y,y'] +\dist[y', \F[C]] < \delta + \frac{\epsilon}{2} < \epsilon,
\]
where $\dist[x,A] := \inf_{a \in A} \dist[x,a]$ for $\emptyset \neq A \subseteq X$ and we have $y \in\F_{\epsilon} \multi {C}$.

So the base case $n=1$ is satisfied for this choice of $\delta$. Assume that for $n$ we have that
\[
 \F_\delta^{\circ n} \multi {\B \delta C} \subseteq \F_\epsilon^{\circ n} \multi {C}.
\]
Applying $\F_\epsilon$ to both sides of this equation and noticing that $\F_\delta \subseteq \F_\epsilon$, we have that
\[
 \F_\delta^{\circ (n+1)} \multi {\B \delta C} \subseteq \F_\epsilon \circ \F_\delta^{\circ n} \multi {\B \delta C} \subseteq \F_\epsilon^{\circ (n+1)} \multi {C}.
\]
Thus, the result holds by induction. ~
\end{proof}

The lemma above allows us to show that the $\epsilon$-chains in the definition of the chain reachable set are allowed to have initial points within $\epsilon$ distance of a point in the initial set.

\begin{prop}\label{prop:chain reach}
Let $\round{X,\dist}$ be a metric space, $\F: X \to X$ be a \acrshort{usc}  multifunction and $C$ a  compact set of $X$.
Then
\begin{dmath*}
\chainreach[\F,C] = \bigcap_{\epsilon >0} \union_{n =0}^\infty \F_\epsilon^{\circ n} \multi C =\bigcap_{\epsilon >0} \union_{n =0}^\infty \F_\epsilon^{\circ n} \multi {\B \epsilon C}.
\end{dmath*}
\end{prop}

\begin{proof} The proof is a straightforward application of the previous lemma. ~
%
\end{proof}

We now can show some necessary conditions on chain reachable robustness.

\begin{thm}\label{thm:necessary continuity}
Let $\round{X,\dist}$ be a compact metric space and $\F:X \multito X$ be a robust \acrshort{usc}  multifunction. Then the following hold:
\begin{enumerate}
\item{$\closure{R}$ is \acrshort{usc}} 
\item{$\closure{R}$ (and ${R}$) is \acrshort{lsc} whenever $\F$ is \acrshort{lsc}}
\end{enumerate}
\end{thm}

\begin{proof}
To prove~(1), notice that by Lemma~\ref{lem:discount usc}, for a given $\delta_1>0$,  there is a $\delta >0$ such that $\reach[\F_\delta,\B \delta x] \subseteq \reach[\F_{\delta_1},x]$ and so by Lemma~\ref{lem:characterized robust} we have for all $\epsilon>0$ there is $\delta>0$ such that
\[
\reach[\F_\delta,\B \delta x] \subseteq \reach_{\frac{\epsilon}{2}} \multi x.
\]
Thus, we also see
\begin{align*}
\closure{\reach}\multi {\B \delta x} \subseteq \closure{\reach[\F_\delta,\B \delta x]} \subseteq \closure{ \reach_{\frac{\epsilon}{2}} \multi x } &\subseteq \reach_{{\epsilon}} \multi x \\
& = \closure{\reach}_{{\epsilon}} \multi x.
\end{align*}
Therefore, by item~(\ref{item:point usc}) of  Proposition~\ref{prop:multi facts}, $\closure{\reach}$ is \acrshort{usc} 

Item~(2) is a restatement of item~(\ref{item:reach lsc}) of Proposition~\ref{prop:multi facts}. ~
\end{proof}

The multifunction $\closure{\reach}$ being continuous (in every sense we discuss here) ends up being a rather strong condition. In particular, it implies some strange things about the minimal sets of $\F$. In Subsection~\ref{subsec:minimal sets}, Propositions~\ref{prop:minimal sets} and~\ref{prop:usc stable} showed that there are minimal sets of $\F$ all of which are Lyapunov stable if $\closure{\reach}$ is \acrshort{usc} and $X$ is compact. When we consider the simpler case of $\F=\set{\f}$ being single-valued with $\f$ continuous and assume all minimal sets are fixed points of $\f$, the fact that \emph{all} the fixed points are Lyapunov stable is already a strong condition. Already, we can tell $\f[x] =x^2$ on $[0,1]$ is not robust since $\bar x = 1$ is not Lyapunov stable. This necessity of Lyapunov stability actually gets stranger. Theorem~\ref{thm:weak basin} gives conditions for the set $\wLsB A$ to be both open and closed. Already we can tell that $\wLsB A$ is closed since $\closure{\reach}$ is \acrshort{usc} But under the assumption that the minimal sets are bounded away from each other we can also show that $\wLsB A$ is open.

\begin{lem}\label{lem:isolated minimal}
Let $\round{X,\dist}$ be a compact metric space and $\F:X \multito X$ be a \acrshort{lsc} multifunction with $\opn{clR}$ being \acrshort{usc} Suppose that $A$ is a minimal set of $\F$ for which there is an open set $V\supseteq A$ such that $V$ contains no other minimal sets except $A$; that is, $A$ is isolated from other minimal sets. Then $\wLsB A$ is open.
\end{lem}

\begin{proof}
Let $x$ be in the open invariant set $\opn{clR}^+ \multi V \supseteq A$. Then by Proposition~\ref{prop:minimal sets} and compactness we know that $\opn{clR} \multi x \subseteq V$ contains a minimal set. This minimal set must be $A$ by the assumption that $A$ is isolated (i.e if $B$ is minimal and $B \subseteq \opn{clR} \multi x $ then $B\subseteq V$ but $A$ is the only minimal set in $V$). Thus, by definition $x\in \wLsB A$ and so $A \subseteq \opn{clR}^+ \multi V \subseteq \wLsB A$. Hence, $\wLsB A$ contains a neighborhood of $A$ and so $\wLsB A$ is open by Proposition~\ref{prop:minimal sets}. ~
\end{proof}

\begin{thm}\label{thm:unique or infinite}
Let $\round{X,\dist}$ be a compact connected metric space and $\F:X \multito X$ is a robust continuous  multifunction. Then either:
\begin{enumerate}
\item{ $\F$ possesses a unique minimal set that is Lyapunov stable.}
\item{ $\F$ possesses infinitely many minimal sets (every minimal set is  Lyapunov stable). Further, for every minimal set $A$ and every open $V\supseteq A$, there is  a minimal set $B$ with $B \subseteq V \setminus A$.}
\end{enumerate}
In the first case, if in addition $\F=\set{\f}$ is single-valued, then the unique minimal set is globally attractive; for all $x \in X$ we have that $\bigcap_{N \in \N} \closure{\set{\f^{\circ n} \round x}_{n =N}^\infty}$ is the unique minimal set.
\end{thm}
\begin{proof}
Either there are finitely many minimal sets or there are infinitely many. Suppose that there are finitely many, and that $A$ is one of these minimal sets. We will show that $\wLsB A$ is a nonempty, closed and open set, then concluding that $\wLsB A=X$ by connectedness. By Theorems~\ref{thm:necessary continuity}~\&~\ref{thm:weak basin}, $\wLsB A$ is closed, it is nonempty since $\emptyset \neq A \subseteq\wLsB A$ and since there are finitely many minimal sets there is an open set of $A$ that contains no other minimal sets. Thus, $\wLsB A$ is also open by Lemma~\ref{lem:isolated minimal} and so $X=\wLsB A$. Now suppose that $B$ is another minimal set of $\F$. Then $b \in B \subseteq \wLsB A$ and by definition of $\wLsB A$ we have that $A \subseteq \closure{\reach}\multi b =B$ (the equality follows from $B$ being minimal and item~(2) of Proposition \ref{prop:minimal sets}). But $B,A$ are minimal, so by definition $A=B$ and $A$ is the unique minimal set in $X$.

In the case where there is at least one isolated minimal set we may apply the above argument. Hence if there are infinitely many minimal sets there can be no isolated minimal sets.
Meaning that, for every minimal set $A$ and open set $V \supseteq A$ there is a minimal set $B$ with $B \subseteq V$ $B\neq A$, since both $B,A$ are minimal if $A\cap B \neq \emptyset$ then $A\cap B$ is a closed nonempty invariant set and $A\cap B \subseteq A$. But $A$ is minimal so $A =A\cap B \subseteq B$ and $B$ is minimal as well, hence, $B=A$ which is a contradiction. Therefore, $A\cap B = \emptyset$ and $B \subseteq V\setminus A$.

If $\F$ is single-valued, the result follows from Proposition~\ref{prop:Lyapunov stable orbit}. \qed
\end{proof}

The above theorem gives us a dramatic dichotomy about the number and properties of the minimal sets of a continuous robust multifunction. In our opinion, the case where $\F=\set{\f}$ is single-valued with all of its minimal sets being fixed points is the easiest case to imagine. In this case it may not be immediately clear if any such functions can satisfy item (2), given the stability requirements on the fixed points. The obvious and easy to forget example of such a function is the identity map. With further assumptions, this is  in fact the only example.

\begin{cor}
Let $X=[a,b]\subseteq \R$, with $a<b$, be equipped with the normal metric. Assume that $\f$ is an analytic function whose minimal sets are all fixed points of $\f$.
If $\f$ is robust, then either $\f$ is the identity function on $X$ or $\f$ has a unique attracting fixed point on $X$.
\end{cor}

\begin{proof}
Suppose that $\f$ is robust, by Theorem~\ref{thm:unique or infinite} there are only two cases: either $\f$ has a unique attracting minimal set or $\f$ has infinitely many minimal sets\textemdash none of which are isolated. By assumption, all these minimal sets are fixed points.
It follows from the identity theorem that an analytic function on a connected and compact set with an infinite number of fixed points is the identity function.

%
\end{proof}

\section{Discussion and Conclusions}\label{sec:conclusions}
Since chain reachable robustness in compact spaces implies that all minimal sets (specifically fixed points and periodic cycles) must be stable, chain reachable robustness is an unusable condition on any dynamics suspected of having unstable behavior; which is a realistic assumption to have when we do not allow for control.
Even if we allow control we should would expect that point to point controllability would not hold if all minimal sets are stable (unless the unique minimal set is the space).

That being said, some ``real'' systems may actually be chain reachable robust and any non-trivial sufficient condition for this would be of interest in order to check for computability of the reachable set. A starting point could be that the functions $\f[x,u]$ are non-expansive functions of $x$ for each $u \in U$, which guarantees the necessary condition $\closure{\reach}$ is \acrshort{usc} and so all the minimal sets of $\F$ would be stable.


\bibliography{biblio}             

\begin{thebibliography}{10}

\bibitem{akshay2016decidable}
{\sc Akshay, S., Genest, B., and H{\'e}lou{\"e}t, L.}
\newblock Decidable classes of unbounded petri nets with time and urgency.
\newblock In {\em International Conference on Application and Theory of Petri
  Nets and Concurrency\/} (2016), Springer, pp.~301--322.

\bibitem{bournez2010robust}
{\sc Bournez, O., Gra{\c{c}}a, D.~S., and Hainry, E.}
\newblock Robust computations with dynamical systems.
\newblock In {\em International Symposium on Mathematical Foundations of
  Computer Science\/} (2010), Springer, pp.~198--208.

\bibitem{chen2015continuous}
{\sc Chen, T., Yu, N., and Han, T.}
\newblock Continuous-time orbit problems are decidable in polynomial-time.
\newblock {\em Information Processing Letters 115}, 1 (2015), 11--14.

\bibitem{collins2005continuity}
{\sc Collins, P.}
\newblock Continuity and computability of reachable sets.
\newblock {\em Theoretical Computer Science 341}, 1-3 (2005), 162--195.

\bibitem{collins2007optimal}
{\sc Collins, P.}
\newblock Optimal semicomputable approximations to reachable and invariant
  sets.
\newblock {\em Theory of Computing Systems 41}, 1 (2007), 33--48.

\bibitem{fan2016automatic}
{\sc Fan, C., Qi, B., Mitra, S., Viswanathan, M., and Duggirala, P.~S.}
\newblock Automatic reachability analysis for nonlinear hybrid models with
  c2e2.
\newblock In {\em International Conference on Computer Aided Verification\/}
  (2016), Springer, pp.~531--538.

\bibitem{fijalkow2019decidability}
{\sc Fijalkow, N., Ouaknine, J., Pouly, A., Sousa-Pinto, J., and Worrell, J.}
\newblock On the decidability of reachability in linear time-invariant systems.
\newblock In {\em Proceedings of the 22nd ACM International Conference on
  Hybrid Systems: Computation and Control\/} (2019), ACM, pp.~77--86.

\bibitem{thesis}
{\sc Fitzsimmons, M.}
\newblock Attractors and semi{-}attractors of {I}{F}{S}.
\newblock Master's thesis, University of Guelph, 50 Stone Rd E, Guelph, ON N1G
  2W1, 4 2018.

\bibitem{gan2017reachability}
{\sc Gan, T., Chen, M., Li, Y., Xia, B., and Zhan, N.}
\newblock Reachability analysis for solvable dynamical systems.
\newblock {\em IEEE Transactions on Automatic Control 63}, 7 (2017),
  2003--2018.

\bibitem{gao2016interpolants}
{\sc Gao, S., and Zufferey, D.}
\newblock Interpolants in nonlinear theories over the reals.
\newblock In {\em International Conference on Tools and Algorithms for the
  Construction and Analysis of Systems\/} (2016), Springer, pp.~625--641.

\bibitem{hu1997handbook}
{\sc Hu, S., and Papageorgiou, N.~S.}
\newblock {\em Handbook of multivalued analysis}, vol.~1.
\newblock Kluwer Dordrecht, 1997.

\bibitem{kong2015dreach}
{\sc Kong, S., Gao, S., Chen, W., and Clarke, E.}
\newblock dreach: $\delta$-reachability analysis for hybrid systems.
\newblock In {\em International Conference on TOOLS and Algorithms for the
  Construction and Analysis of Systems\/} (2015), Springer, pp.~200--205.

\bibitem{lal2019counterexample}
{\sc Lal, R., and Prabhakar, P.}
\newblock Counterexample guided abstraction refinement for polyhedral
  probabilistic hybrid systems.
\newblock {\em ACM Transactions on Embedded Computing Systems (TECS) 18}, 5s
  (2019), 98.

\bibitem{li2018robustly}
{\sc Li, Y., and Liu, J.}
\newblock Robustly complete reach-and-stay control synthesis for switched
  systems via interval analysis.
\newblock In {\em 2018 Annual American Control Conference (ACC)\/} (2018),
  IEEE, pp.~2350--2355.

\bibitem{li2018rocs}
{\sc Li, Y., and Liu, J.}
\newblock Rocs: A robustly complete control synthesis tool for nonlinear
  dynamical systems.
\newblock In {\em Proceedings of the 21st International Conference on Hybrid
  Systems: Computation and Control (part of CPS Week)\/} (2018), ACM,
  pp.~130--135.

\bibitem{liu2017robust}
{\sc Liu, J.}
\newblock Robust abstractions for control synthesis: Completeness via
  robustness for linear-time properties.
\newblock In {\em Proceedings of the 20th International Conference on Hybrid
  Systems: Computation and Control\/} (2017), pp.~101--110.

\bibitem{rungger2018accurate}
{\sc Rungger, M., and Zamani, M.}
\newblock Accurate reachability analysis of uncertain nonlinear systems.
\newblock In {\em Proceedings of the 21st International Conference on Hybrid
  Systems: Computation and Control (part of CPS Week)\/} (2018), ACM,
  pp.~61--70.

\end{thebibliography}
                    \bibliographystyle{acm}                                 








\end{document}